\documentclass[10pt]{article}
\usepackage[letterpaper,twoside,outer=1.3in,vmargin=1.3in,]{geometry}
\usepackage[ruled,vlined]{algorithm2e}
\usepackage{fullpage}
\usepackage{enumerate}
\usepackage{color}
\usepackage{paralist}
\usepackage{float}
\usepackage{tikz,booktabs}
\usetikzlibrary{decorations.pathreplacing}

\usepackage{amsmath,amsfonts,amsthm,amssymb,bbm}
\def \Sprime {Q}
\def \bR {\mathbb{R}}

\def \R {\mathbb{R}}

\def \conv {\text{conv}}

\def \cS {\mathcal{S}}
\def \bcS {\bar{\mathcal{S}}}
\def \XJDd {X^{l,u}}

\def \DJDd {\Delta^{l,u}}

\def \bXJD {\bar{X}^{0,u}}%^{J,D}}
\def \bXJDd {\bar{X}^{l,u}}
\def \bXJd {\bar{X}^{l,n}}
\def \bDJD {\bar{\Delta}^{0,u}}%^{J,D}}
\def \bDJDd {\bar{\Delta}^{l,u}}
\def \bDJd {\bar{\Delta}^{l,n}}

\def \goodfamily {proper} %{$D$\text{-complete}}
\newcommand{\keywords}[1]{\textbf{\textit{Keywords---}} #1}

\DeclareMathOperator{\proj}{proj}

\newtheorem{thm}{Theorem}
\newtheorem{lem}[thm]{Lemma}
\newtheorem{prop}[thm]{Proposition}

\newtheorem{rmk}[thm]{Remark}

\newtheorem{defn}[thm]{Definition}

\title{Multilinear Sets with Two Monomials and Cardinality Constraints}
\author{Rui Chen \\ University of Wisconsin-Madison \\[-.2cm]\small (rchen234@wisc.edu) \and 
	Sanjeeb Dash\\ IBM Research \\[-.2cm]\small (sanjeebd@us.ibm.com) \and
	Oktay G\"{u}nl\"{u}k \\ Cornell University \\[-.2cm] \small (ong5@cornell.edu)}
\date{\today}

\begin{document}
\maketitle\allowdisplaybreaks
\begin{abstract}
  Binary polynomial optimization is equivalent to the problem of minimizing a linear function over the intersection of the multilinear set with a polyhedron.
  Many families of valid inequalities for the multilinear set are available in the literature, though giving a polyhedral characterization of the convex hull is not tractable in general as binary polynomial optimization is NP-hard.
  In this paper, we study the cardinality constrained multilinear set in the special case when the number of monomials is exactly two. We give an extended formulation, with two more auxiliary variables and exponentially many inequalities, of the convex hull of solutions of the standard linearization of this problem. We also show that the separation problem can be solved efficiently. %Earlier, Crama and Rodr{\'\i}guez-Heck (2017) gave a polyhedral description of the convex hull of the multilinear set when there are two monomials.
\end{abstract}
\keywords{binary polynomial optimization, cardinality constraint, polyhedral combinatorics}

\section{Introduction}
%%\begin{linenomath}
Binary polynomial optimization problems, which consist of minimizing a polynomial objective function of binary variables subject to polynomial constraints, are often formulated as integer programming problems, and solved using standard off-the-shelf MIP solvers.
Consider a polynomial function of a vector $x \in \R^n$ given by
\begin{displaymath}
f(x)=\beta+\sum_{i=1}^m \gamma_i\prod_{j\in S_i} x_j^{\alpha_{ij}},
\end{displaymath}
where $S_i \subseteq \{1, \ldots, n\}$ for $i=1,\ldots, m$, $\alpha_{ij}$ is a positive integer for $j \in S_i$, $\beta\in\R$, and $\gamma\in\R^m$.
If $x_j\in\{0,1\}$, then $x^k_j=x_j$ for any integer $k\ge1$. Therefore, $f(x)$ is equal to the multilinear function $\beta+ \sum_{i=1}^m\gamma_i \prod_{j \in S_i} x_j$ (no variable appears with an exponent greater than 1) for all $x \in \{0,1\}^n$.
One can then replace each product $\prod_{j\in S_i} x_j$ with a new variable $\delta_i$ and the {\em standard linearization constraints} 
%--- $\delta_i\leq x_j$ for $j\in S_i$ and $\delta_i\geq 1-\sum_{j\in S_i}(1-x_j)$. 
$$\delta_i\leq x_j\text{~for~}j\in S_i\text{~and~}\delta_i\geq 1-\sum_{j\in S_i}(1-x_j).$$ 
These constraints enforce the condition that when $x \in \{0,1\}^n$, then $\delta \in \{0,1\}^m$ and $\delta_i = \prod_{j \in S_i} x_j$, and therefore
	\[ f(x)=\beta+\sum_{i=1}^m\gamma_i \delta_i \mbox{ if } x\in \{0,1\}^n. \]
This is the basis of Fortet's result \cite{fortet} which states that every binary polynomial optimization problem can be formulated as an integer programming problem where all discrete variables are restricted to be binary. The constraints in this integer program that involve the relationship between the $\delta$ and $x$ variables form the {\em multilinear set}, which is defined as the set of binary points $(\delta,x)\in\{0,1\}^{m+n}$  subject to the constraints that each $\delta_i$ variable is the product of certain $x_j$ variables:
\begin{displaymath}
Y = \Big\{(\delta,x)\in\{0,1\}^{m+n}:\delta_i=\prod_{j\in S_i}x_j,~i=1,\ldots,m\Big\}.
\end{displaymath}
Unconstrained binary polynomial optimization is equivalent to optimizing a linear function over the multilinear set.
The convex hull of the multilinear set is called the {\em multilinear polytope} \cite{del2017polyhedral}, and several classes of valid inequalities for the multilinear polytope have been proposed recently \cite{crama2017class,del2017polyhedral,del2018decomposability,del2018multilinear,del2021running}. The {\em boolean quadric polytope} \cite{padberg1989boolean} is a multilinear polytope when each $|S_i| = 2$, i.e., $\delta_i$ is exactly a product of two distinct $x_j$ variables. The \emph{maximum monomial agreement} problem has been studied in the context of machine learning \cite{dbs, eg2, eg3, dgm}, and can be formulated as optimizing a linear function over a multilinear set.
%\end{linenomath}

The polyhedral structure of $\conv(Y)$ has been fully characterized in some cases. When the multilinear set is defined by a single nonlinear monomial, i.e., $m = 1$, then it is known (see, e.g., \cite{crama1993concave}) that the standard linearization of $Y$ gives $\conv(Y)$. When the nonlinear monomials have a nested structure, i.e., the sets $S_i$ have a nested structure, then the convex hull is obtained by augmenting the standard linearization constraints with the 2-link inequalities \cite{crama2017class} (call the associated polytope the {\em 2-link polytope}). This result follows from the work of Fischer, Fischer and McCormick \cite{fischer2018matroid}. Independently, Crama and Rodr\'iguez-Heck \cite{crama2017class} showed that the 2-link inequalities are facet-defining in the nested case. They also showed that the 2-link polytope is equal to $\conv(Y)$ when $m = 2$. Del Pia and Khajavirad gave complete descriptions of $\conv(Y)$ for multilinear sets associated with certain acyclic hypergraphs \cite{del2018multilinear,del2021running}.

%\begin{linenomath}
The boolean quadric polytope with an upper bound constraint on the number of nonzero $x_j$ variables was studied in \cite{mehrotra1997cardinality}. Optimization over the cardinality constrained boolean quadric polytope has applications in, for example, the maximum edge weight clique problem \cite{sorensen2004new,fomeni2017new,hosseinian2017maximum}.
In \cite{chen2020cardinality,chen2020multilinear}, Chen, Dash and G\"{u}nl\"{u}k studied the cardinality constrained multilinear set (CCMS), defined as\begin{displaymath}
X = \Big\{(\delta,x)\in\{0,1\}^{m+n}:\delta_i=\prod_{j\in S_i}x_j,~i=1,\ldots,m,~L\leq\sum_{j=1}^nx_j\leq U\Big\},
\end{displaymath}
where $L,U$ are nonnegative integers.
%The problem of minimizing a linear function over $X$ contains as a special case the {\em maximum monomial agreement} problem, and a cardinality constrained version of it, which has been analyzed in the context of machine learning \cite{dbs, eg2, eg3, dgm}, and solved via branch-and-bound methods and heuristics. 
%Facetial properties of the convex hulls of the CCMS were characterized for such sets satisfying a set of conditions called properness. Using properness,
%
Optimizing a linear function over $X$ appears as a pricing subproblem in \cite{dash2018boolean}, where Dash, G\"{u}nl\"{u}k and Wei applied column generation to solve a binary classification problem.
An explicit polyhedral description of the convex hull when the sets $S_i$ are nested was given in \cite{chen2020cardinality}. We note that this description is significantly more complicated than the 2-link polytope.
When $L=0$, the above result also follows from the results of Fischer, Fischer and McCormick \cite{fischer2018matroid}, who gave a polyhedral description of the convex hull of $V = \{(\delta, x) \in Y : x \in \mathcal{M}\}$, where $\mathcal{M}$ is the independent set polytope of a matroid over $n$ elements.
When the matroid is a uniform matroid, $V$ is the same as $X$ with $L = 0$.
The results in \cite{fischer2018matroid} generalize earlier results of  Buchheim and Klein \cite{bk} and Fischer and Fischer \cite{ff}, who gave a polyhedral description for the quadratic minimum spanning tree problem with a single quadratic term in the objective. % -- of the form $c_{ef}x_ex_f$ where $e, f$ are edges in a graph, $x_e, x_f$ are 0-1 variables (with value 1 if $e,f$ are present in a spanning tree solution, and $c_{ef}$ is a real number.
%\end{linenomath}

In the above papers, a tractable variant of the multilinear set is obtained by imposing a structure on the sets $S_i$ (e.g., nestedness), or bounding the number of such sets (e.g., $m = 1$ in \cite{bk} and \cite{ff}).
Motivated by the work of Crama and Rodr\'iguez-Heck \cite{crama2017class}, we consider the special case of the CCMS when $m=2$ (call it {\em CCMS-2}) and characterize its convex hull.
If we project the general CCMS onto the space of $x$ variables and any pair of $\delta_i$ variables, we get an example of CCMS-2, and thus valid inequalities for CCMS-2 yield valid inequalities for the general CCMS, and can be used to strengthen linear relaxations of cardinality constrained binary polynomial optimization problems.

%\begin{linenomath}
For the sake of convenience, we apply a change of variables by letting $z_j=1-x_j$ for $j=1,\ldots,n$, $l=n-U$ and $u=n-L$, which leads to an affine transformation of the set $X$ into\begin{displaymath}
\Big\{(\delta,z)\in\{0,1\}^{m+n}:\delta_i=\prod_{j\in S_i}(1-z_j),~i\in I,~l\leq\sum_{j\in J}z_j\leq u\Big\}.
\end{displaymath}
We still call this set the CCMS. In the rest of the paper, we will work with the CCMS in this form mainly because convex hull description is visually more appealing (the inequalities look simpler).
%\end{linenomath}

The paper is organized as follows. In Section \ref{sec:prelim}, we review related results. We then give an extended formulation of CCMS-2 by introducing two auxiliary sets $S_i$ and associated $\delta_i$ variables and giving exponentially many valid inequalities. In Section \ref{sec:special}, we show that these valid inequalities give the convex hull for the special case when $S_1\cap S_2\neq\emptyset$ and $|S_1\cup S_2|\leq n-l$. In Section \ref{sec:separation}, we show that the separation problem for the given valid inequalities can be solved efficiently. In Section \ref{sec:general}, we generalize our results to the case when the conditions $S_1\cap S_2\neq\emptyset$ or $|S_1\cup S_2|\leq n-l$ do not necessarily hold.

\section{Preliminaries}\label{sec:prelim}
%\begin{linenomath}
We next formally define the CCMS, give an integer programming formulation for it and review some concepts and results from  \cite{chen2020cardinality,chen2020multilinear}. 
%\oo{Let $I=\{1,\ldots,m\}$ and $J=\{1,\ldots,n\}$
Let $I$ and $J$ be two finite sets with  $|I|=m$ and $|J|=n$ and let $\cS:=\{S_i\}_{i\in I}$ be a family of distinct subsets of $J$ with $|S_i|\ge 1$ for  $i\in I$. 
Given two integers $l,u\ge 0$  such that $u\geq 2$ and $ u\ge l+1$, the CCMS associated with family $\cS$ can be formulated as follows:
\begin{displaymath}
	\XJDd(\cS):=\Big\{(\delta,z)\in\{0,1\}^{m+n}:\delta_i=\prod_{j\in S_i}(1-z_j),~i\in I,~l\leq\sum_{j\in J}z_j\leq u \Big\}.
\end{displaymath}
For simplicity, we will refer to this set as $\XJDd$ when the associated family $\cS$ is clear from the context.
Note that if $|S_i|> n-l$ then $\delta_i=0$ for all feasible $(\delta,z)\in \XJDd$, and we therefore can assume that $|S_i|\le n-l$ for all $i\in I$.
Let $\DJDd := \proj_{\delta}\XJDd$ denote the orthogonal projection of $\XJDd$ onto the space of $\delta$ variables.
%\end{linenomath}

We next define a property of $\cS$ called {\em properness} which has been used in \cite{chen2020multilinear} to give a complete characterization of  $\conv(\XJDd)$ when $\cS$ is nested.

\begin{defn}\label{DCdef}
	A family $\cS=\{S_i\}_{i\in I}$ of subsets of $J$ is called a \textit{{\goodfamily} family} (with respect to $l$ and $u$) if it satisfies the following properties:\begin{enumerate}
		%		\item $\DJDd$ is a set of exactly $m+1$ affinely independent vectors in $\bR^{m}$;
		\item $\DJDd$ is a set of exactly $|\cS|+1$ affinely independent points in $\bR^{|\cS|}$;
		\item $\cS$ is closed under nonempty intersection.
	\end{enumerate} 
\end{defn}
%Some useful results can be established for $\XJDd$ associated with a proper family.
The following result combines Lemmas 3 and 5 in \cite{chen2020multilinear}.

\begin{thm}[Chen et al. \cite{chen2020multilinear}]\label{thm:facet}
	Assume $\cS$ is a proper family (with respect to $l$ and $u$). Then each facet of $\conv(\XJDd)$ can be defined by an inequality of the form $\alpha^Tz+\beta^T\delta\leq\gamma$ where one of the following two is true:\begin{enumerate}
		\item[(i)] {\bf Type-1 inequalities: }$\alpha=0$ and $\beta^T\delta\leq\gamma$ defines a facet of $\conv(\DJDd)$;
		\item[(ii)] {\bf Type-2 inequalities: }$\alpha\in\{0,\kappa \}^{n}$ for some $\kappa\in\{+1,-1\}$ and 
		\begin{equation}
			\gamma-\beta^T\bar{\delta}=\max_{z:(\bar{\delta},z)\in\XJDd}\alpha^Tz\label{type2def}
		\end{equation}
		for all $\bar{\delta}\in\DJDd$. 
	\end{enumerate}
	Moreover, any inequality satisfying $(ii)$ is valid for $\conv(\XJDd)$.
\end{thm}

Note that when  $\cS$  is  a {{\goodfamily} family},  $\conv(\DJDd)$ is a simplex with exactly $|\cS|+1$ facets and therefore it is straightforward to describe all type-1 facets of $\conv(\XJDd)$  when  $\cS$  is  a {{\goodfamily} family}.

\subsection{An extended formulation for CCMS-2}
We now focus on the case when $\cS=\{S_1,S_2\}$ and $n-l\ge|S_i|\geq 2$ for $i\in \{1,2\}$. If $\cS$ is nested (i.e., $S_1\subset S_2$ or $S_1\supset S_2$), then the linear inequality description of $\conv(\XJDd(\cS))$ has been given in \cite{chen2020cardinality,chen2020multilinear}.
(Also see \cite{crama2017class} for the special case when $u=n$ and $l=0$.)
We therefore consider the case when $S_1\setminus S_2\neq\emptyset$ and $S_2\setminus S_1\neq\emptyset$.

%\begin{linenomath}
To make use of Theorem \ref{thm:facet}, we define the extended family of subsets  $\bar\cS=\{S_0,S_1,S_2,S_3\}$ where $S_0:=S_1\cap S_2$ and $S_3:=S_1\cup S_2$.
Note that the extended family $\bar\cS$ is closed under nonempty intersection and therefore satisfies one of the conditions of properness given in Definition \ref{DCdef}. We next define two 
%Instead of studying $\conv(\XJDd(\cS))$ in the original space of its variables, we will study  extended formulations of it with 
additional variables $\delta_0$ and $\delta_3$ associated with the sets $S_0$ and $S_3$ as follows:
$$\delta_0=\prod_{j\in S_1\cap S_2}(1-z_j),~~~\delta_3=\prod_{j\in S_1\cup S_2}(1-z_j).$$
%\end{linenomath}

%\begin{linenomath}
Note that even though we assumed $n-l\ge|S_i|\geq 2$ for $i=1,2$, it is possible that the new sets $S_0$ and $S_3$ would not satisfy this property. 
We therefore use the convention $\delta_0=1$ when $|S_0|=0$ (i.e.if $S_1\cap S_2=\emptyset$).
Also note that if $|S_1\cup S_2|>n-l$ then $\delta_3=0$ due to the lower bound on the sum of the $z$ variables. 
From now on we will focus on the the extended family $\bar\cS=\{S_0,S_1,S_2,S_3\}$ and  note that a linear inequality description of $\conv(\XJDd(\bar{\cS}))$ would give an extended formulation for $\conv(\XJDd(\cS))$. 
To simplify notations, we denote $\XJDd(\bar{\cS})$ by $\bar{X}^{l,u}$, and its orthogonal projection onto the $\delta$ space by $\bDJDd\subseteq\{0,1\}^4$. 
We next present a {\em naive} formulation for the  $\bar{X}^{l,u}$ and then present valid inequalities for it which would lead to a complete linear inequality description of its convex hull:
\begin{align*}\bar{X}^{l,u}=\Big\{z\in\{0,1\}^{|J|},\delta\in\{0,1\}^4:		~&z_j+\delta_i\leq1,~ j\in S_i,~ i\in\{0,1,2,3\};\\
	&\delta_i+\sum_{j\in S_i}z_j\ge1,~i\in\{0,1,2,3\};~u\ge\sum_{j\in J}z_j\ge l\Big\}.
\end{align*}
%\end{linenomath}

\subsection{Type-1 inequalities for $\bar{X}^{l,u}$ and a new formulation}
%\begin{linenomath}
When $\bar\cS$ is a proper family, Theorem \ref{thm:facet} establishes that all type-1 inequalities that are facet-defining for $\conv(\XJDd(\bar{\cS}))$ have to be one of the inequalities that defines a facet  of $\conv(\bDJDd)$. %The next lemma will be used later to describe $\conv(\bDJDd)$. 
Note that   $S_0 \subset S_1,S_2 $ implies that $\delta_0\ge \delta_1,\delta_2$, similarly  $S_1,S_2 \subset S_3$ implies that $\delta_1,\delta_2\ge\delta_3$, and $\delta_3=1$ if $\delta_1=\delta_2=1$ for all feasible $\delta\in\bDJDd$. Consequently, $\bDJDd\subseteq \Delta $ where
%and therefore all $\delta\in \bDJDd$  can be characterized as follows:
\begin{equation}\label{set:bDelta0}
	\Delta := \Big\{(0,0,0,0),(1,0,0,0),(1,1,0,0),(1,0,1,0),(1,1,1,1)\Big\}.
\end{equation}
Also note that $\bDJDd$ is strictly contained in $ \Delta $ when at least one of the following two conditions hold:\begin{enumerate}
	\item[$(i)$] If $S_1\cap S_2=\emptyset$, then $(0,0,0,0)\not\in\bDJDd$, as $\delta_0=1$ for all $\delta\in\bDJDd$;
	\item[$(ii)$] If $|S_1\cup S_2|>n-l$, then $(1,1,1,1)\not\in\bDJDd$, as $\delta_3=0$ for all $\delta\in\bDJDd$.
\end{enumerate}
If either $S_1\cap S_2=\emptyset$ or $|S_1\cup S_2|>n-l$, then $\bar\cS$ is not a proper family as $\bDJDd$ contains fewer than $|\cS|+1=5$ points.
We next give an inequality description of the convex hull of $\Delta$.
%\end{linenomath}

\begin{lem}\label{lem:simplex}
	The following inequalities give the convex hull of $\Delta$.\begin{align}
		\delta_0\leq&~1,\label{d0le1}\\
		-\delta_3\leq&~0,\label{d3ge0}\\
		\delta_3\leq&~\delta_i,~~~~~~~~~~~~~~~~~~i\in\{1,2\},\label{d3ledi}\\
		\delta_1+\delta_2\leq&~\delta_0+\delta_3.\label{dineq}
	\end{align}
\end{lem}
\begin{proof}
	It is easy to check that the inequalities above are satisfied by all $\delta\in \Delta$ (given in equation \eqref{set:bDelta0}) and therefore they are valid for their convex hull.
	To see that the 5 inequalities above give the convex hull of $\Delta$, note that $\Delta$ consists of 5 points in $\R^4$ that define a full-dimensional simplex. Inequalities \eqref{d0le1}-\eqref{dineq} precisely define the facets of that simplex.
	%shown in \eqref{set:bDelta0}.
	%	Validity of \eqref{d0le1}-\eqref{d3ledi} is trivial. Inequality \eqref{dineq} can be derived by observing the following valid inequality\begin{displaymath}
	%	\prod_{j\in S_0}(1-z_j)\prod_{i=1}^2\Big(1-\prod_{j\in S_i\setminus S_0}(1-z_j)\Big)\geq 0,
	%	\end{displaymath}
	%	which is equivalent to\begin{displaymath}
	%	\prod_{j\in S_1}(1-z_j)+\prod_{j\in S_2}(1-z_j)\leq \prod_{j\in S_0}(1-z_j)+\prod_{j\in S_3}(1-z_j),
	%	\end{displaymath}
	%	i.e., \eqref{dineq}.
\end{proof}
Note that as $\bDJDd\subseteq \Delta$, inequalities \eqref{d0le1}-\eqref{dineq} are valid for $\bDJDd$.

%\begin{linenomath}
We next present the  {\em standard linearization inequalities} which,  together with inequalities \eqref{d0le1}-\eqref{dineq} and integrality constraints give a reformulation of $\bar{X}^{l,u}$:
\begin{align}
	\sum_{j\in J}z_j\leq &~u,\label{cardconu}\\
	z_j+\delta_i\leq &~1, &&j\in S_i,~i\in\{0,1,2\},\label{zjdile1}\\
	z_j\leq &~1, &&j\in J,\label{zjle1}\\
	1-\delta_0-\sum_{j\in S_0}z_j\leq &~0,\label{zjdige1}\\
	\delta_0-\delta_i-\sum_{j\in S_i\setminus S_0}z_j\leq&~0,&&i\in\{1,2\},\label{2link}\\
	-z_j\leq &~0, &&j\in J,\label{zjge0}\\
	-\sum_{j\in J}z_j\leq&-l,\label{cardconl}
\end{align}
It is easy to see that inequalities \eqref{cardconu}-\eqref{cardconl} are valid for $\bXJDd$, and together with inequalities \eqref{d0le1}-\eqref{dineq} they imply all the inequalities used in the naive formulation. Consequently,  the inequality system  \eqref{d0le1}-\eqref{cardconl} together with integrality constraints gives a valid integer linear  description of $\bXJDd$.
We note that inequalities \eqref{2link} are referred to as 2-link inequalities in \cite{crama2017class}. 
%\end{linenomath}

\subsection{Type-2 inequalities for $\bar{X}^{l,u}$}
In this section we present two sets of inequalities and show that they are valid for $\bar{X}^{l,u}$.
We will later show that these inequalities subsume all  facet-defining type-2 inequalities for $\bar{X}^{l,u}$ that are different from the standard linearization inequalities \eqref{cardconu}-\eqref{cardconl}.
This would lead to an inequality description of $\conv(\bXJDd)$ provided that $\bDJDd=\Delta$.

The first set of inequalities that we present below are in fact valid for $\bXJD$ as they do not depend on the lower bound $l$.
Also note that the coefficient of  the $z_j$ variables in these %following set of 
inequalities are either $0$ or $+1$ and they all have a right-hand side of $u$.
%\oo{As $\bXJDd\subseteq\bXJD$, these inequalities are also valid for $\bXJDd$.}
\begin{prop}\label{prop:ineqs1}
	Let $\Sprime \subseteq J$, the following inequalities are valid for $\bXJD(\supseteq \bXJDd)$:
	\begin{align}
		\sum_{j\in \Sprime }z_j+(u-|\Sprime \setminus S_0|)\delta_0+\sum_{i=1}^2(|\Sprime \cap S_i\setminus S_0|)\delta_i\leq &~u&&\text{if }|\Sprime \setminus S_0|\leq u,\label{eq:ineq1}\\
		\sum_{j\in \Sprime }z_j+\sum_{i=1}^2(u-|\Sprime \setminus S_i|)\delta_i+(|\Sprime \setminus S_0|-u)\delta_3\leq&~u,\label{eq:ineq2}\\
		\sum_{j\in \Sprime }z_j+\delta_0-\delta_i+(u-1-|\Sprime \setminus S_k|)\delta_k+\delta_3\leq&~u,&&\text{if }Q\cap(S_i\setminus S_0)=\emptyset,~\{i,k\}=\{1,2\},\label{eq:ineq3}\\
		\sum_{j\in \Sprime }z_j+(u-|\Sprime \setminus S_i|)\delta_i+(|\Sprime \cap S_3\setminus S_i|)\delta_3\leq&~u&&\text{if }|\Sprime \setminus S_i|\leq u,~i\in\{1,2\},\label{eq:ineq4}\\
		\sum_{j\in \Sprime }z_j+(u-|\Sprime \setminus S_3|)\delta_3\leq&~u&&\text{if }|\Sprime \setminus S_3|\leq u.\label{eq:ineq5}
	\end{align}
\end{prop}
\begin{proof} Let $(\delta,z)\in \bXJD$.
	We first show validity of \eqref{eq:ineq2} using the possible values of $\delta\in \bDJD\subseteq \Delta$ shown in \eqref{set:bDelta0}. If $\delta_3=1$, then $z_j=0$ for $j\in S_3$, $\delta=(1,1,1,1)$ and \eqref{eq:ineq2} becomes $\sum_{j\in \Sprime \setminus S_3}z_j\leq|\Sprime \setminus S_3|$, which is implied by \eqref{zjle1}. If $(\delta_1,\delta_2,\delta_3)=\mathbf{0}$, then \eqref{eq:ineq2} is implied by \eqref{cardconu} and \eqref{zjge0}. If $(\delta_1,\delta_2,\delta_3)=(1,0,0)$, then $z_j=0$ for $j\in S_1$ and \eqref{eq:ineq2} becomes $\sum_{j\in \Sprime \setminus S_1}z_j\leq |\Sprime \setminus S_1|$, which is implied by \eqref{zjle1}. Similarly, if $(\delta_1,\delta_2,\delta_3)=(0,1,0)$, then \eqref{eq:ineq2} is implied by \eqref{zjle1}.
	
	Similarly, we next show validity of \eqref{eq:ineq3}. Without loss of generality, assume $i=1$ and $k=2$. If $\delta_0=0$, then $\delta=(0,0,0,0)$ and \eqref{eq:ineq3} is implied by \eqref{cardconu} and \eqref{zjge0}. If $\delta_3=1$, then $z_j=0$ for $j\in S_3$, $\delta=(1,1,1,1)$ and \eqref{eq:ineq3} becomes $\sum_{j\in \Sprime \setminus S_3}z_j\leq |\Sprime \setminus S_2|$, which is implied by \eqref{zjle1}. If $\delta=(1,1,0,0)$, then \eqref{eq:ineq3} is implied by \eqref{cardconu} and \eqref{zjge0}. If $\delta=(1,0,1,0)$, then $z_j=0$ for $j\in S_2$ and \eqref{eq:ineq3} becomes $\sum_{j\in \Sprime \setminus S_2}z_j\leq |\Sprime \setminus S_2|$, which is implied by \eqref{zjle1}. If $\delta=(1,0,0,0)$, then \eqref{eq:ineq3} becomes $\sum_{j\in Q}z_j\leq u-1$, which is implied by \eqref{cardconu}, \eqref{zjge0} and the fact that $\sum_{j\in S_2\setminus S_0}z_j\geq 1$ if $(\delta_0,\delta_2)=(1,0)$.
	
	Validity of inequalities \eqref{eq:ineq1}, \eqref{eq:ineq4} and \eqref{eq:ineq5} follows from the validity of the generalized inequalities in \cite[Proposition 20]{chen2020multilinear}.
\end{proof}

We now present the second set of valid inequalities  that have the property that the coefficient of the $z_j$ variables are either $-1$ or $0$ and they all have a right-hand side of $0$. %\oo{As $\bXJDd\subseteq\bXJd$, these inequalities are also valid for $\bXJDd$.}
Also note that these inequalities are valid for $\bXJd$ as they do not depend on the upper bound $u$.

\begin{prop}\label{prop:ineqs2}
	For any $\Sprime \subseteq J$, the following inequalities are valid for $\bXJd(\supseteq\bXJDd)$:
	\begin{align}
		-\sum_{j\in \Sprime }z_j+(l+|\Sprime \cup S_0|-n)\delta_0+\sum_{i=1}^2(|S_i\setminus S_0\setminus \Sprime |)\delta_i\leq&~0&&\text{if }|\Sprime \cup S_0|\geq n-l,\label{eq:ineq6}\\
		-\sum_{j\in \Sprime }z_j+\sum_{i=1}^2(l+|\Sprime \cup S_i|-n)\delta_i+(n-|\Sprime \cup S_0|-l)\delta_3\leq&~0,\label{eq:ineq7}\\
		-\sum_{j\in \Sprime }z_j+\delta_0-\delta_i+(l+|\Sprime \cup S_k|-1-n)\delta_k+\delta_3\leq&~0,&&\text{if }S_i\setminus S_0\subseteq Q,~\{i,k\}=\{1,2\},\label{eq:ineq8}\\
		-\sum_{j\in \Sprime }z_j+(l+|\Sprime \cup S_i|-n)\delta_i+(|S_3\setminus S_i\setminus \Sprime |)\delta_3\leq&~0&&\text{if }|\Sprime \cup S_i|\geq n-l,~i\in\{1,2\},\label{eq:ineq9}\\
		-\sum_{j\in \Sprime }z_j+(l+|\Sprime \cup S_3|-n)\delta_3\leq&~0&&\text{if }|\Sprime \cup S_3|\geq n-l.\label{eq:ineq10}
	\end{align}
\end{prop}
\begin{proof}
	Let $(\delta,z)\in \bXJd$. 
	We first show validity of \eqref{eq:ineq7}  by considering possible values of $\delta\in \bDJd$.
	If $\delta_3=1$, then $z_j=0$ for $j\in S_3$, $\delta_3=(1,1,1,1)$ and \eqref{eq:ineq7} is equivalent to $-\sum_{j\in \Sprime \cup S_3}z_j\leq n-l-|\Sprime \cup S_3|$, which is implied by \eqref{cardconl} and \eqref{zjle1}. If $(\delta_1,\delta_2,\delta_3)=(0,0,0)$, then \eqref{eq:ineq7} is implied by \eqref{zjge0}. If $(\delta_1,\delta_2,\delta_3)=(1,0,0)$, then $z_j=0$ for $j\in S_1$ and \eqref{eq:ineq7} is equivalent to $-\sum_{j\in \Sprime \cup S_1}z_j\leq n-l-|\Sprime \cup S_1|$, which is implied by \eqref{cardconl} and \eqref{zjle1}.  Similarly, if $(\delta_1,\delta_2,\delta_3)=(0,1,0)$, then $z_j=0$ for $j\in S_2$ and \eqref{eq:ineq7} is implied by \eqref{cardconl} and \eqref{zjle1}.
	
	We next show validity of \eqref{eq:ineq8}. Without loss of generality, assume $i=1$ and $k=2$. If $\delta_0=0$, then $\delta=(0,0,0,0)$ and \eqref{eq:ineq8} is implied by \eqref{zjge0}. If $\delta_3=1$, then $z_j=0$ for $j\in S_3$, $\delta=(1,1,1,1)$ and \eqref{eq:ineq8} is equivalent to $-\sum_{j\in \Sprime \cup S_2}z_j\leq n-l-|\Sprime \cup S_2|$, which is implied by \eqref{cardconl} and \eqref{zjle1}. If $\delta=(1,1,0,0)$, then \eqref{eq:ineq8} is implied by \eqref{zjge0}. If $\delta=(1,0,1,0)$, then $z_j=0$ for $j\in S_2$ and \eqref{eq:ineq8} is equivalent to $-\sum_{j\in \Sprime \cup S_2}z_j\leq n-l-|\Sprime \cup S_2|$, which is implied by \eqref{cardconl} and \eqref{zjle1}. If $\delta=(1,0,0,0)$, then \eqref{eq:ineq8} becomes $-\sum_{j\in Q}z_j\leq -1$, which is implied by \eqref{zjge0} and the fact that $\sum_{j\in S_2\setminus S_0}z_j\geq 1$ if $(\delta_0,\delta_2)=(1,0)$.
	
	Validity of inequalities \eqref{eq:ineq6}, \eqref{eq:ineq9} and \eqref{eq:ineq10} follows from the validity of the generalized inequalities in \cite[Proposition 21]{chen2020multilinear}.
\end{proof}

In the remainder of this paper, we will show that inequalities \eqref{d0le1}-\eqref{eq:ineq10} that we have presented in this section are sufficient to describe $\conv(\bXJDd)$. 

\section{Special case: $S_1\cap S_2\neq\emptyset$ and $|S_1\cup S_2|\leq n-l$}\label{sec:special}
%\begin{linenomath}
We first study the ``easy" case when $\DJDd=\Delta$, i.e., $S_1\cap S_2\neq\emptyset$ and $|S_1\cup S_2|\leq n-l$. In this case, \begin{equation}\label{set:bDelta}
	\bDJDd=\Delta=\{(0,0,0,0),(1,0,0,0),(1,1,0,0),(1,0,1,0),(1,1,1,1)\}
\end{equation}
is a set of $|\bar{S}|+1$ affinely independent points. Then $\bcS$ is a proper family (with respect to $l$ and $u$).
We can therefore use Theorem \ref{thm:facet} to characterize the facet-defining inequalities for  $\conv(\bXJDd)$.
We start with establishing the conditions under which $\conv(\bXJDd)$ is full-dimensional and describe its affine hull when it is not.
%\end{linenomath}

\begin{prop}\label{prop:affine_hull}
	Assume $S_0\neq\emptyset$ and $|S_3|\leq n-l$. If $|S_1\cap S_2|=|S_0|> 1$, then $\conv(\bXJDd)$ is full-dimensional. Otherwise, the affine hull of $\conv(\bXJDd)$ is defined by equality $z_{j_0}+\delta_0=1$ where $\{j_0\}=S_1\cap S_2$, which is implied by inequalities \eqref{zjdile1} and \eqref{zjdige1}.
\end{prop}
\begin{proof}
	Let $n_i:=|S_i|$ for $i\in\{0,1,2,3\}$. Then $n_1+n_2=n_0+n_3$. Without loss of generality, assume that $S_1=\{1,2,\ldots,n_1\}$ and $S_2=\{n_1-n_0+1,n_1-n_0+2,\ldots,n_3 \}$. Then $S_0=\{n_1-n_0+1,\ldots,n_1 \}$. Let $\alpha^Tz+\beta^T\delta=\gamma$ be an equality that is satisfied by all $(\delta,z)\in\bXJDd$.
	
	Given any $Q\subseteq J$ with $l\leq|Q|\leq u$, we define the point $(\delta^Q,z^Q)\in\bXJDd$ such that\begin{displaymath}
	z_j^Q=\left\{\begin{array}{ll}
	1, &\text{if }j\in Q\\
	0, &\text{otherwise}
	\end{array}\right.,~~~~~~~~~\delta_i^Q=\prod_{j\in S_i}(1-z_j^Q).
	\end{displaymath}
	
	Fix $t\in S_0$. For each $j\in J\setminus\{t\}$, let $Q_j$ be a subset of $J$ such that $\{j,t\}\subseteq Q_j$ and $|Q_j|=u$. Then $\delta^{Q_j}=\mathbf{0}$ since $t\in Q_j\cap S_0$. By definition of $(\alpha,\beta,\gamma)$, we have $\alpha^Tz^{Q_j}+\beta^T\mathbf{0}=\gamma$. Define $\bar{U}_j=Q_j\setminus\{j\}$. Then $|\bar{U}_j|=u-1\in[l,u]$ and $\delta^{\bar{U}_j}=\mathbf{0}$. Therefore, $\alpha^Tz^{\bar{U}_j}+\beta^T\mathbf{0}=\gamma$. Together with $\alpha^Tz^{Q_j}+\beta^T\mathbf{0}=\gamma$, we have $\alpha_j=0$ for each $j\in J\setminus\{t\}$. Then $\alpha_{t}=\alpha^Tz^{Q_j}+\beta^T\mathbf{0}=\gamma$. We consider two separate cases depending on the value of $|S_0|$:\begin{enumerate}
		\item If $|S_0|=n_0\geq 2$, then letting $t=n_1$ and $t=n_1-1$ in the above construction implies $\alpha=\mathbf{0}$. Since $\conv(\bDJDd)=\proj_{\delta}\conv(\bXJDd)$ is a full-dimensional simplex, no equality of the form $\beta^T\delta=\gamma$ with nonzero $(\beta,\gamma)$ can be valid. Therefore, $\conv(\bXJDd)$ is full-dimensional if $|S_0|\geq 2$.
		\item If $|S_0|=1$, then we can only have $t=j_0=n_1$. Equality $\alpha^Tz+\beta^T\delta=\gamma$ becomes $\gamma z_{j_0}+\beta^T\delta=\gamma$. Since $\conv(\bDJDd)=\proj_{\delta}\conv(\bXJDd)$ is full-dimensional, there can be at most one nonzero equality (up to nonzero scaling) of the form $z_{j_0}+\beta^T\delta=1$ that defines the affine hull of $\conv(\bXJDd)$. This nonzero equality can only be $z_{j_0}+\delta_0=1$.
	\end{enumerate}
\end{proof}

By Lemma \ref{lem:simplex}, %Since $\conv(\bDJDd)$ is a simplex, it is easy to verify that 
inequalities \eqref{d0le1}-\eqref{dineq} define all facets of $\conv(\bDJDd)=\conv(\Delta)$, which are type-1 inequalities of $\conv(\bXJDd)$. By Theorem \ref{thm:facet}, it remains to describe all type-2 inequalities of $\conv(\bXJDd)$. 

\subsection{Type-2 inequalities with $\alpha\ge0$}\label{sec:typege0}
Before presenting the type-2 inequalities, we list some valid inequalities that are implied by the standard linearization and inequalities \eqref{eq:ineq1}-\eqref{eq:ineq5}. These inequalities are used later for describing all type-2 inequalities with $\alpha\geq 0$.
\begin{lem}\label{lem:redundant_ineq}
	%\begin{linenomath}	
	Assume $S_0\neq\emptyset$ and $|S_3|\leq n-l$. Then for any $Q\subseteq J$, the following inequalities are implied by inequalities \eqref{cardconu}-\eqref{zjge0} and \eqref{eq:ineq1}-\eqref{eq:ineq5}:
	\begin{align}
		\sum_{j\in Q}z_j+(|Q\cap S_0|)\delta_0+(|Q\cap(S_1\setminus S_0)|)\delta_1+(|Q\cap(S_2\setminus S_0)|)\delta_2\leq&~|Q|,\label{redundant1u}\\
		\sum_{j\in Q}z_j\leq&~u,\label{redundant2u}\\
		\sum_{j\in Q\setminus(S_i\setminus S_0)}z_j+\delta_0-\delta_i\leq&~u, &&i\in\{1,2\},\label{redundant3u}\\
		\sum_{j\in Q\setminus(S_3\setminus S_0)}z_j+2\delta_0-\delta_1-\delta_2+(u-|Q\setminus S_3|)\delta_3\leq&~u, &&\text{if }|Q\setminus S_3|\leq u,\label{redundant4u}\\
		\sum_{j\in Q\setminus(S_3\setminus S_0)}z_j+2\delta_0-\delta_1-\delta_2\leq&~u,\label{redundant5u}\\
		\sum_{j\in Q\setminus S_3}z_j+\delta_0-\delta_1-\delta_2+(u-|Q\setminus S_3|)\delta_3\leq&~u-1, &&\text{if }|Q\setminus S_3|\leq u.\label{redundant6u}
	\end{align}
	%\end{linenomath}
\end{lem}
\begin{proof}
	%\begin{linenomath}
	We can rewrite inequality \eqref{redundant1u} as\begin{multline*}
	\sum_{j\in Q}z_j+(|Q\cap S_0|)\delta_0+(|Q\cap(S_1\setminus S_0)|)\delta_1+(|Q\cap(S_2\setminus S_0)|)\delta_2\\
	=\sum_{j\in Q\cap S_0}\underbrace{(z_j+\delta_0)}_{\leq 1}+\sum_{j\in Q\cap(S_1\setminus S_0)}\underbrace{(z_j+\delta_1)}_{\leq 1}+\sum_{j\in Q\cap(S_2\setminus S_0)}\underbrace{(z_j+\delta_2)}_{\leq 1}+\sum_{j\in Q\setminus S_3}\underbrace{z_j}_{\leq 1}\leq|Q|,
	\end{multline*}
	which is implied by inequalities \eqref{zjdile1} and \eqref{zjle1}. 
	%\end{linenomath}
	
	%\begin{linenomath}
	Inequality \eqref{redundant2u} is trivially implied by \eqref{cardconu} and \eqref{zjge0}. 
	Inequality \eqref{redundant3u} can be derived from inequalities \eqref{cardconu}, \eqref{2link} and \eqref{zjge0} as\begin{displaymath}
	\sum_{j\in Q\setminus(S_i\setminus S_0)}z_j+\delta_0-\delta_i=\underbrace{\sum_{j\in J}z_j}_{\leq u}-\sum_{j\in J\setminus (Q\cup(S_i\setminus S_0))}\underbrace{z_j}_{\geq 0}+\underbrace{\delta_0-\delta_i-\sum_{j\in S_i\setminus S_0}z_j}_{\leq 0}\leq u.
	\end{displaymath}
	Inequality \eqref{redundant4u} can be rewritten as\begin{multline*}
	\sum_{j\in Q\setminus(S_3\setminus S_0)}z_j+2\delta_0-\delta_1-\delta_2+(u-|Q\setminus S_3|)\delta_3\\
	=\underbrace{\sum_{j\in Q\cup(S_3\setminus S_0)}z_j+(u-|Q\setminus S_3|)\delta_3}_{\leq u}+\sum_{i=1}^2\underbrace{(\delta_0-\delta_i-\sum_{j\in S_i\setminus S_0}z_j)}_{\leq 0}\leq u,
	\end{multline*}
	which is implied by \eqref{2link} and \eqref{eq:ineq5} with $Q$ replaced by $Q\cup(S_3\setminus S_0)$.
	%\end{linenomath}
	
	Inequality \eqref{redundant5u} can be rewritten as\begin{displaymath}
	\sum_{j\in Q\setminus(S_3\setminus S_0)}z_j+2\delta_0-\delta_1-\delta_2=\underbrace{\sum_{j\in J}z_j}_{\leq u}-\sum_{j\in J\setminus (Q\cup (S_3\setminus S_0))}\underbrace{z_j}_{\geq 0}+\sum_{i=1}^2\underbrace{(\delta_0-\delta_i-\sum_{j\in S_i\setminus S_0}z_j)}_{\leq 0}\leq u,
	\end{displaymath}
	which is implied by \eqref{cardconu}, \eqref{2link} and \eqref{zjge0}.
	
	Inequality \eqref{redundant6u} can be rewritten as\begin{multline*}
	\sum_{j\in Q\setminus S_3}z_j+\delta_0-\delta_1-\delta_2+(u-|Q\setminus S_3|)\delta_3\\
	=\underbrace{\sum_{j\in Q\cup S_3}z_j+(u-|Q\setminus S_3|)\delta_3}_{\leq u}+\sum_{i=1}^2\underbrace{(\delta_0-\delta_i-\sum_{j\in S_i\setminus S_0}z_j)}_{\leq 0}-\underbrace{(\delta_0+\sum_{j\in S_0}z_j)}_{\geq 1}\leq u-1,
	\end{multline*}
	which is implied by \eqref{zjdige1}, \eqref{2link} and \eqref{eq:ineq5} with $Q$ replaced by $Q\cup S_3$.
\end{proof}

We are now ready to give all type-2 inequalities with $\alpha\geq 0$.

\begin{lem}\label{lem:type2ge0}
	Assume $S_0\neq\emptyset$ and $|S_3|\leq n-l$ and let $\alpha^Tz+\beta^T\delta\leq\gamma$ be a facet-defining type-2 inequality for $\conv(\bXJDd)$. If $\alpha\in\{0,1\}^n$, then the inequality is implied by inequalities \eqref{cardconu}-\eqref{zjge0} and \eqref{eq:ineq1}-\eqref{eq:ineq5}.
\end{lem}
\begin{proof}
	%\begin{linenomath}
	By Lemma \ref{lem:redundant_ineq}, inequalities \eqref{redundant1u}-\eqref{redundant6u} are implied by inequalities \eqref{cardconu}-\eqref{zjge0} and \eqref{eq:ineq1}-\eqref{eq:ineq5}. Therefore, we only need to show that inequality $\alpha^Tz+\beta^T\delta\leq\gamma$ is implied by inequalities \eqref{cardconu}-\eqref{zjge0}, \eqref{eq:ineq1}-\eqref{eq:ineq5} and \eqref{redundant1u}-\eqref{redundant6u}.
	Let $T$ denote the set $\{j\in J:\alpha_j=1 \}$ and let $t_0\in\{0,1\}$ denote if $T\cap S_0=\emptyset$ or not, and $t_i\in\{0,1\}$ denote if $T\cap (S_i\setminus S_0)=\emptyset$ or not for $i\in\{1,2\}$. More precisely,\begin{displaymath}
	t_0=\begin{cases}	1& \text{if }T\cap S_0=\emptyset\\0& \text{otherwise}	\end{cases},~~~t_1=\begin{cases}	1& \text{if }T\cap (S_1\setminus S_0)=\emptyset\\0& \text{otherwise}	\end{cases},~~~t_2=\begin{cases}	1& \text{if }T\cap (S_2\setminus S_0)=\emptyset\\0& \text{otherwise}	\end{cases}.
	\end{displaymath}
	According to the definition of type-2 inequalities, given \eqref{set:bDelta} we can write equations \eqref{type2def} as\begin{equation}\label{eq:type2u}\begin{aligned}
	\gamma=\min\{u-t_0,|T|\},\qquad&
	\gamma-\beta_0=\min\{u-t_1-t_2,|T\setminus S_0|\},\\
	\gamma-\beta_0-\beta_1=	\min\{u-t_2,|T\setminus S_1|\},\qquad&
	\gamma-\beta_0-\beta_2=
	\min\{u-t_1,|T\setminus S_2|\},\\
	\gamma-\beta_0-\beta_1-\beta_2-\beta_3&=\min\{u,|T\setminus S_3|\}.
	\end{aligned}
	\end{equation}
	%\end{linenomath}

	We first consider the case when $T\cap S_0\neq\emptyset$, i.e., $t_0=0$. Without loss of generality, assume that $|T\cap(S_1\setminus S_0)|\geq |T\cap(S_2\setminus S_0)|$. Then $|T\setminus S_1|\leq|T\setminus S_2|$ and $t_1\leq t_2$. We now consider three subcases depending on the value of $t_1+t_2$ assuming $t_0=0$:
	
	%\begin{linenomath}
	{\bf  Case 1:} $t_0=t_1=t_2=0$. Depending on the cardinality, based on \eqref{eq:type2u}, inequality $\alpha^Tz+\beta^T\delta\leq\gamma$ becomes$$\begin{cases} 
		\text{inequality \eqref{redundant1u} with $\Sprime = T$ } & \text{if }u\geq|T|,\\
		\text{inequality \eqref{eq:ineq1} with $\Sprime =T$ } & \text{if } |T|>u\geq |T\setminus S_0|,\\
		\text{inequality \eqref{eq:ineq2} with $\Sprime =T$ } & \text{if } |T\setminus S_0|>u\geq|T\setminus S_2|,\\
		\text{inequality \eqref{eq:ineq4} with  $\Sprime =T,~ i=1$ } & \text{if } |T\setminus S_2|>u\geq |T\setminus S_1|,\\
		\text{inequality \eqref{eq:ineq5} with $\Sprime =T$ } & \text{if } |T\setminus S_1|>u\geq |T\setminus S_3|,\\
		\text{inequality \eqref{redundant2u} with $\Sprime=T$ } & \text{if }|T\setminus S_3|>u.
	\end{cases}$$
	{\bf Case 2: }$t_0=0$, $t_1+t_2=1$. Since $t_1\leq t_2$, in this case we have $t_1=0$ and $t_2=1$. Therefore, $T\cap(S_2\setminus S_0)=\emptyset$. Depending on the cardinality, based on \eqref{eq:type2u}, inequality $\alpha^Tz+\beta^T\delta\leq\gamma$ becomes$$\begin{cases}
	\text{inequality \eqref{redundant1u} with $\Sprime = T$ } & \text{if }u-1\geq|T|,\\
	\text{inequality \eqref{eq:ineq1} with $\Sprime =T$ } & \text{if }|T|\geq u\geq |T\setminus S_0|,\\
	\text{inequality \eqref{eq:ineq3} with $\Sprime =T,~ (i,k)=(2,1)$ } & \text{if }|T\setminus S_0|\geq u>|T\setminus S_1|,\\
	\text{inequality \eqref{redundant3u} with $\Sprime =T,~ i=2$ } & \text{if }|T\setminus S_1|\geq u.
	\end{cases}$$
	{\bf Case 3: }$t_0=0$, $t_1=t_2=1$. In this case, $T\cap(S_3\setminus S_0)=\emptyset$. Depending on the cardinality, based on \eqref{eq:type2u}, inequality $\alpha^Tz+\beta^T\delta\leq\gamma$ becomes$$\begin{cases}
	\text{inequality \eqref{redundant1u} with $\Sprime = T$ } & \text{if }u-1\geq|T|,\\
	\text{inequality \eqref{eq:ineq1} with $\Sprime =T$ } & \text{if }|T|\geq u\text{ and }|T\setminus S_0|\leq u-2,\\
	\text{inequality \eqref{redundant4u} with $\Sprime =T$ } & \text{if }|T\setminus S_0|=u-1,\\
	\text{inequality \eqref{redundant5u} with $\Sprime =T$ } & \text{if }|T\setminus S_0|\geq u.
	\end{cases}$$
	%\end{linenomath}
	
	%\begin{linenomath}
	Next we consider the case when $T\cap S_0=\emptyset$. If $|T|=u-1$ and $T\cap S_3\neq \emptyset$, or if $|T|\leq u-2$, then by \eqref{eq:type2u} we have $\gamma=\gamma-\beta_0=|T|$, $\gamma-\beta_0-\beta_1=|T\setminus S_1|$, $\gamma-\beta_0-\beta_2=|T\setminus S_2|$, $\gamma-\beta_0-\beta_1-\beta_2-\beta_3=|T\setminus S_3|$. In this case, inequality $\alpha^Tz+\beta^T\delta\leq\gamma$ becomes inequality \eqref{redundant1u} with $\Sprime =T$. If $|T|=u-1$ and $T\cap S_3=\emptyset$, then inequality $\alpha^Tz+\beta^T\delta\leq\gamma$ becomes inequality \eqref{redundant6u} with $\Sprime =T$.
	\if0
	Inequalities \eqref{zjdile1} and \eqref{zjle1} imply that\begin{displaymath}
	\alpha^Tz+\beta^T\delta=\sum_{j\in T\cap S_1}\underbrace{(z_j+\delta_1)}_{\leq 1}+\sum_{j\in T\cap S_2}\underbrace{(z_j+\delta_2)}_{\leq 1}+\sum_{j\in T\setminus S_3}\underbrace{z_j}_{\leq 1} \leq|T|=\gamma.
	\end{displaymath}
	If $|T|=u-1$ and $T\cap S_3=\emptyset$, then\begin{displaymath}
	\sum_{j\in T}z_j+\delta_0-\delta_1-\delta_2+\delta_3=\alpha^Tz+\beta^T\delta\leq\gamma=u-1
	\end{displaymath}
	is implied by inequalities \eqref{zjdige1}, \eqref{2link} and \eqref{eq:ineq5} with $\Sprime =T\cup S_3$ as\begin{displaymath}
	\alpha^Tz+\beta^T\delta=\sum_{j\in T}z_j+\delta_0-\delta_1-\delta_2+\delta_3=\underbrace{(\sum_{j\in T\cup S_3}z_j+\delta_3)}_{\leq u}+\underbrace{(-\delta_0-\sum_{j\in S_0}z_j)}_{\leq -1}+\sum_{i=1}^2\underbrace{(\delta_0-\delta_i-\sum_{j\in S_i\setminus S_0}z_j)}_{\leq 0}\leq u-1=\gamma.
	\end{displaymath}
	\fi
	Otherwise, $|T|\geq u$.	Note that $(T,\beta,\gamma)$ satisfies equations in \eqref{eq:type2u}. By \eqref{eq:type2u} and Theorem \ref{thm:facet}, when $|T|\geq u$, the following inequality is a valid type-2 inequality\begin{equation}\label{type2:u}
	\sum_{j\in T\cup S_0}z_j+(\beta_0+1)\delta_0+\beta_1\delta_1+\beta_2\delta_2+\beta_3\delta_3\leq u=\gamma+1.
	\end{equation}
	Since $(T\cup S_0)\cap S_0\neq\emptyset$, \eqref{type2:u} is implied by \eqref{cardconu}-\eqref{zjge0} and \eqref{eq:ineq1}-\eqref{eq:ineq5} according to our discussion of cases when $t_0=0$.
	Inequality \eqref{type2:u} together with \eqref{zjdige1} implies $\alpha^Tz+\beta^T\delta\leq\gamma$. We thus conclude $\alpha^Tz+\beta^T\delta\leq\gamma$ is implied by \eqref{cardconu}-\eqref{zjge0} and \eqref{eq:ineq1}-\eqref{eq:ineq5}.
	%\end{linenomath}
\end{proof}

\subsection{Type-2 inequalities with $\alpha\le0$}
Similar to Section \ref{sec:typege0}, in order to describe all type-2 inequalities with $\alpha\leq 0$, we first list some valid inequalities that are implied by standard linearization and inequalities \eqref{eq:ineq6}-\eqref{eq:ineq10}.
\begin{lem}\label{lem:redundant_ineq2}
	%\begin{linenomath}
	Assume $S_0\neq\emptyset$ and $|S_3|\leq n-l$. Then for any $Q\subseteq J$, the following inequalities are implied by inequalities \eqref{zjdile1}-\eqref{cardconl} and \eqref{eq:ineq6}-\eqref{eq:ineq10}:
	\begin{align}
		-\sum_{j\in Q}z_j+(|S_0\setminus Q|)\delta_0+(|S_1\setminus S_0\setminus Q|)\delta_1+(|S_2\setminus S_0\setminus Q|)\delta_2\leq&~n-|Q|-l,\label{redundant1l}\\
		-\sum_{j\in Q}z_j\leq&~0,\label{redundant2l}\\
		-\sum_{j\in Q\cup(S_i\setminus S_0)}z_j+\delta_0-\delta_i\leq&~0, &&i\in\{1,2\},\label{redundant3l}\\
		-\sum_{j\in Q\cup(S_3\setminus S_0)}z_j+2\delta_0-\delta_1-\delta_2+(l+|Q\cup S_3|-n)\delta_3\leq&~0, &&\text{if }|Q\cup S_3|\geq n-l,\label{redundant4l}\\
		-\sum_{j\in Q\cup(S_3\setminus S_0)}z_j+2\delta_0-\delta_1-\delta_2\leq&~0,\label{redundant5l}\\
		-\sum_{j\in Q\cup S_3}z_j+\delta_0-\delta_1-\delta_2+(l+|Q\cup S_3|-n)\delta_3\leq&~-1, &&\text{if }|Q\cup S_3|\geq n-l.\label{redundant6l}
	\end{align}
	%\end{linenomath}
\end{lem}
\begin{proof}
	%\begin{linenomath}
	We can rewrite inequality \eqref{redundant1l}  as\begin{multline*}
		-\sum_{j\in Q}z_j+(|S_0\setminus Q|)\delta_0+(|S_1\setminus S_0\setminus Q|)\delta_1+(|S_2\setminus S_0\setminus Q|)\delta_2\\
		=-\underbrace{\sum_{j\in J}z_j}_{\geq l}+\sum_{j\in S_0\setminus Q}\underbrace{(z_j+\delta_0)}_{\leq 1}+\sum_{j\in S_1\setminus S_0\setminus Q}\underbrace{(z_j+\delta_1)}_{\leq 1}+\sum_{j\in S_2\setminus S_0\setminus Q}\underbrace{(z_j+\delta_2)}_{\leq 1}+\sum_{j\in J\setminus(Q\cup S_3)}\underbrace{z_j}_{\leq 1}
		\leq n-|Q|-l,
	\end{multline*}
	which is implied by inequalities \eqref{zjdile1}, \eqref{zjle1} and \eqref{cardconl}. Inequality  \eqref{redundant2l} is trivially implied by  \eqref{zjge0}.
	%\end{linenomath}
	
	%\begin{linenomath}
	Inequality \eqref{redundant3l} can be derived from inequalities \eqref{2link} and \eqref{zjge0} as\begin{displaymath}
		-\sum_{j\in Q}z_j+\delta_0-\delta_i=-\sum_{j\in Q\setminus(S_i\setminus S_0)}\underbrace{z_j}_{\geq 0}+\underbrace{\delta_0-\delta_i-\sum_{j\in S_i\setminus S_0}z_j}_{\leq 0}\leq 0.
	\end{displaymath}
	Inequality \eqref{redundant4l} can be rewritten as\begin{multline*}
		-\sum_{j\in Q\cup(S_3\setminus S_0)}z_j+2\delta_0-\delta_1-\delta_2+(l+|Q\cup S_3|-n)\delta_3\\
		=\underbrace{-\sum_{j\in Q\setminus(S_3\setminus S_0)}z_j+(l+|Q\cup S_3|-n)\delta_3}_{\leq 0}+\sum_{i=1}^2\underbrace{(\delta_0-\delta_i-\sum_{j\in S_i\setminus S_0}z_j)}_{\leq 0}\leq 0,
	\end{multline*}
	which is implied by \eqref{2link} and \eqref{eq:ineq10} with $Q$ replaced by $Q\setminus(S_3\setminus S_0)$.
	%\end{linenomath}
	
	%\begin{linenomath}
	Inequality \eqref{redundant5l} can be rewritten as\begin{displaymath}
		-\sum_{j\in Q\cup(S_3\setminus S_0)}z_j+2\delta_0-\delta_1-\delta_2=-\sum_{j\in Q\setminus(S_3\setminus S_0)}\underbrace{z_j}_{\geq 0}+\sum_{i=1}^2\underbrace{(\delta_0-\delta_i-\sum_{j\in S_i\setminus S_0}z_j)}_{\leq 0}\leq 0,
	\end{displaymath}
	which is implied by \eqref{2link} and \eqref{zjge0}.
	%\end{linenomath}

	%\begin{linenomath}
	Inequality \eqref{redundant6l} can be rewritten as\begin{multline*}
		-\sum_{j\in Q\cup S_3}z_j+\delta_0-\delta_1-\delta_2+(l+|Q\cup S_3|-n)\delta_3\\
		=\underbrace{-\sum_{j\in Q\setminus S_3}z_j+(l+|Q\cup S_3|-n)\delta_3}_{\leq 0}+\sum_{i=1}^2\underbrace{(\delta_0-\delta_i-\sum_{j\in S_i\setminus S_0}z_j)}_{\leq 0}-\underbrace{(\delta_0+\sum_{j\in S_0}z_j)}_{\geq 1}\leq-1,
	\end{multline*}
	which is implied by \eqref{zjdige1}, \eqref{2link} and \eqref{eq:ineq10} with $Q$ replaced by $Q\setminus S_3$.
	%\end{linenomath}
\end{proof}

We now apply Lemma \ref{lem:redundant_ineq2} to show that all type-2 inequalities with $\alpha\leq 0$ are implied by inequalities \eqref{zjdile1}-\eqref{cardconl} and \eqref{eq:ineq6}-\eqref{eq:ineq10}.

\begin{lem}\label{lem:type2le0}
	Assume $S_0\neq\emptyset$ and $|S_3|\leq n-l$ and let $\alpha^Tz+\beta^T\delta\leq\gamma$ be a facet-defining type-2 inequality for $\conv(\bXJDd)$. If $\alpha\in\{0,-1\}^n$, then the inequality is implied by inequalities \eqref{zjdile1}-\eqref{cardconl} and \eqref{eq:ineq6}-\eqref{eq:ineq10}.
\end{lem}
\begin{proof}
	%\begin{linenomath}
	By Lemma \ref{lem:redundant_ineq2}, inequalities \eqref{redundant1l}-\eqref{redundant6l} are implied by inequalities \eqref{zjdile1}-\eqref{cardconl} and \eqref{eq:ineq6}-\eqref{eq:ineq10}. Therefore, we only need to show that inequality $\alpha^Tz+\beta^T\delta\leq\gamma$ is implied by inequalities \eqref{zjdile1}-\eqref{cardconl}, \eqref{eq:ineq6}-\eqref{eq:ineq10} and \eqref{redundant1l}-\eqref{redundant6l}. Let $T$ denote the set $\{j\in J:\alpha_j=-1 \}$ and let $t_0\in\{0,1\}$ denote if $S_0\subseteq T$, and $t_i\in\{0,1\}$ denote if $S_i\setminus S_0\subseteq T$ for $i\in\{1,2\}$. More precisely,\begin{displaymath}
	t_0=\begin{cases}	1& \text{if }S_0\subseteq T\\0& \text{otherwise}	\end{cases},~t_1=\begin{cases}	1& \text{if }S_1\setminus S_0\subseteq T\\0& \text{otherwise}	\end{cases},~t_2=\begin{cases}	1& \text{if }\text{if }S_2\setminus S_0\subseteq T\\0& \text{otherwise}	\end{cases}.
	\end{displaymath}
	According to the definition of type-2 inequalities, given \eqref{set:bDelta} we can write equations \eqref{type2def} as\begin{equation}\label{eq:type2l}
	\begin{aligned}
	\gamma=-t_0-(l-t_0-|J\setminus T|)^+=&-t_0-(l+|T|-n-t_0)^+,\\
	\gamma-\beta_0=-t_1-t_2-(l-t_1-t_2-|J\setminus S_0\setminus T|)^+=&-t_1-t_2-(l+|T\cup S_0|-n-t_1-t_2)^+,\\
	\gamma-\beta_0-\beta_1=-t_2-(l-t_2-|J\setminus S_1\setminus T|)^+=&-t_2-(l+|T\cup S_1|-n-t_2)^+,\\
	\gamma-\beta_0-\beta_2=-t_1-(l-t_1-|J\setminus S_2\setminus T|)^+=&-t_1-(l+|T\cup S_2|-n-t_1)^+,\\
	\gamma-\beta_0-\beta_1-\beta_2-\beta_3=-(l-|J\setminus S_3\setminus T|)^+=&-(l+|T\cup S_3|-n)^+.
	\end{aligned}
	\end{equation}
	Here $a^+$ denotes $\max\{0,a\}$ for $a\in \bR$.
	%\end{linenomath}
	
	We first consider the case when $S_0\nsubseteq T$, i.e., $t_0=0$. Without loss of generality, assume that $|S_1\setminus S_0\setminus T|\geq|S_2\setminus S_0\setminus T|$. Then $|T\cup S_1|\geq|T\cup S_2|$ and $t_1\leq t_2$. We now consider three subcases depending on the value of $t_1+t_2$ assuming $t_0=0$:
	
	%\begin{linenomath}
	{\bf Case 1:} $t_0=t_1=t_2=0$. Depending on the cardinality, based on \eqref{eq:type2l}, inequality $\alpha^Tz+\beta^T\delta\leq\gamma$ becomes$$\begin{cases}  
	\text{inequality \eqref{redundant1l} with $\Sprime =T$ } & \text{if } n-l\leq |T|, \\
	\text{inequality \eqref{eq:ineq6} with $\Sprime =T$ } & \text{if } |T|<n-l\leq |T\cup S_0|, \\
	\text{inequality \eqref{eq:ineq7} with $\Sprime =T$ } & \text{if }|T\cup S_0|<n-l\leq |T\cup S_2|, \\%(\leq |T\cup S_1|)
	\text{inequality \eqref{eq:ineq9} with  $\Sprime =T,~ i=1$ } & \text{if } |T\cup S_2|<n-l\leq |T\cup S_1|, \\
	\text{inequality \eqref{eq:ineq10} with $\Sprime =T$ } & \text{if } |T\cup S_1|<n-l\leq|T\cup S_3|, \\
	\text{inequality \eqref{redundant2l} with $\Sprime =T$ } & \text{if } |T\cup S_3|<n-l.
	\end{cases}$$
	{\bf Case 2:} $t_0=0$, $t_1+t_2=1$. Since $t_1\leq t_2$, in this case we have $t_1=0$ and $t_2=1$. Therefore, $S_2\setminus S_0\subseteq T$. Depending on the cardinality, based on \eqref{eq:type2l}, inequality $\alpha^Tz+\beta^T\delta\leq\gamma$ becomes$$\begin{cases}  
	\text{inequality \eqref{redundant1l} with $\Sprime =T$ } & \text{if } n-l+1\leq |T|, \\
	\text{inequality \eqref{eq:ineq6} with $\Sprime =T$ } & \text{if } |T|\leq n-l<|T\cup S_0|, \\
	\text{inequality \eqref{eq:ineq8} with $\Sprime =T,~ (i,k)=(2,1)$ } & \text{if }|T\cup S_0|<n-l\leq |T\cup S_1|, \\
	\text{inequality \eqref{redundant3l} with $\Sprime =T,~ i=2$ } & \text{if } |T\cup S_1|<n-l.
	\end{cases}$$
	{\bf Case 3:} $t_0=0$, $t_1=t_2=1$. In this case, $S_3\setminus S_0\subseteq T$. Depending on the cardinality, based on \eqref{eq:type2l}, inequality $\alpha^Tz+\beta^T\delta\leq\gamma$ becomes$$\begin{cases}  
	\text{inequality \eqref{redundant1l} with $\Sprime =T$ } & \text{if } n-l+1\leq |T|, \\
	\text{inequality \eqref{eq:ineq6} with $\Sprime =T$ } & \text{if } |T|\leq n-l\text{ and }|T\cup S_0|\geq n-l+2, \\
	\text{inequality \eqref{redundant4l} with $\Sprime =T$ } & \text{if }|T\cup S_0|=n-l+1, \\
	\text{inequality \eqref{redundant5l} with $\Sprime =T$ } & \text{if } |T\cup S_0|\leq n-l.
	\end{cases}$$
	%\end{linenomath}
	
	%\begin{linenomath}
	Next we consider the case when $S_0\subseteq T$. If $|T|=n-l+1$ and $S_3\nsubseteq T$, or if $|T|\geq n-l+2$, then by \eqref{eq:type2l} we have $\gamma=\gamma-\beta_0=n-|T|-l,\gamma-\beta_0-\beta_1=n-|S_1\cup T|-l,\gamma-\beta_0-\beta_2=n-|S_2\cup T|-l,\gamma-\beta_0-\beta_1-\beta_2-\beta_3=n-|S_3\cup T|-l$. In this case, inequality $\alpha^Tz+\beta^T\delta\leq\gamma$ becomes inequality \eqref{redundant1l} with $\Sprime =T$.  If $|T|=n-l+1$ and $S_3\subseteq T$, then inequality $\alpha^Tz+\beta^T\delta\leq\gamma$ becomes inequality \eqref{redundant6l} with $\Sprime =T$.
\if0	
	Then inequalities \eqref{zjdile1}, \eqref{zjle1} and \eqref{cardconl} imply that\begin{align*}
	\alpha^Tz+\beta^T\delta=&-\sum_{j\in T}z_j+|S_1\setminus T|\delta_1+|S_2\setminus T|\delta_2\\
	=&\sum_{j\in J\setminus(T\cup S_3)}\underbrace{z_j}_{\leq 1}-\underbrace{\sum_{j\in J}z_j}_{\geq l}+\sum_{j\in S_1\setminus T}\underbrace{(z_j+\delta_1)}_{\leq 1}+\sum_{j\in S_2\setminus T}\underbrace{(z_j+\delta_2)}_{\leq 1}\\
	\leq&~n-|T\cup S_3|-l+|S_1\setminus T|+|S_2\setminus T|=|J\setminus T|-l=\gamma.
	\end{align*}
	If $|T|=n-l+1$ and $S_3\subseteq T$, then\begin{displaymath}
	-\sum_{j\in T}z_j+\delta_0-\delta_1-\delta_2+\delta_3=\alpha^Tz+\beta^T\delta\leq\gamma=-1
	\end{displaymath}
	is implied by \eqref{zjdige1}, \eqref{2link} and \eqref{eq:ineq10} with $\Sprime =T\setminus S_3$ as\begin{displaymath}
	\alpha^Tz+\beta^T\delta=-\sum_{j\in T}z_j+\delta_0-\delta_1-\delta_2+\delta_3=\underbrace{(-\sum_{j\in T\setminus S_3}z_j+\delta_3)}_{\leq 0}+\underbrace{(-\delta_0-\sum_{j\in S_0}z_j)}_{\leq -1}+\sum_{i=1}^2\underbrace{(\delta_0-\delta_i-\sum_{j\in S_i\setminus S_0}z_j)}_{\leq 0}\leq -1=\gamma.
	\end{displaymath}
\fi
	Otherwise, $|T|\leq n-l$.
	Note that $(T,\beta,\gamma)$ satisfies equations in \eqref{eq:type2l}. By \eqref{eq:type2l} and Theorem \ref{thm:facet}, when $|T|\leq n-l$, the following inequality is a valid type-2 inequality\begin{equation}\label{type2:l}
	-\sum_{j\in T\setminus S_0}z_j+(\beta+1)\delta_0+\beta_1\delta_1+\beta_2\delta_2+\beta_3\delta_3\leq 0=\gamma+1.
	\end{equation}
	Since $S_0\nsubseteq T\setminus S_0$, \eqref{type2:l} is implied by \eqref{zjdile1}-\eqref{cardconl} and \eqref{eq:ineq6}-\eqref{eq:ineq10} according to our discussion of cases when $t_0=0$. Inequality \eqref{type2:l} together with \eqref{zjdige1} implies $\alpha^Tz+\beta^T\delta\leq\gamma$. We thus conclude $\alpha^Tz+\beta^T\delta\leq\gamma$ is implied by \eqref{cardconu}-\eqref{zjge0} and \eqref{eq:ineq1}-\eqref{eq:ineq5}.
	%\end{linenomath}
\end{proof}

Combining Theorem \ref{thm:facet}, Proposition \ref{prop:affine_hull}, Lemmas \ref{lem:simplex}, \ref{lem:type2ge0}, and \ref{lem:type2le0}, we obtain a complete description of $\conv(\bXJDd)$.

\begin{thm}\label{thm:1}
	If $S_0\neq\emptyset$ and $|S_3|\leq n-l$, then $\conv(\bXJDd)$ is given by inequalities \eqref{d0le1}-\eqref{eq:ineq10}.
\end{thm}
\begin{proof}
	Since $\bcS$ is a proper family, by Theorem \ref{thm:facet}, a complete inequality description of $\conv(\bXJDd)$ consists of equalities that describe the affine hull of $\bXJDd$, and type-1 and type-2 inequalities of $\conv(\bXJDd)$. By Proposition \ref{prop:affine_hull}, the affine hull of $\bXJDd$ is either equal to $\bR^{m+n}$ or described by inequality \eqref{zjdile1} and \eqref{zjdige1}. By Lemma \ref{lem:simplex}, all type-1 are given by inequalities \eqref{d0le1}-\eqref{dineq}. By Lemmas \ref{lem:type2ge0} and \ref{lem:type2le0}, all type-2 inequalities are implied by inequalities \eqref{cardconu}-\eqref{eq:ineq10}.
\end{proof}
\begin{rmk}\label{rmk:1}
	We note that Theorem \ref{thm:1} still holds even when $l=u$. It follows from the fact that we can reduce the $l=u$ case to the $l<u$ case. Since $|J\setminus S_3|=n-|S_3|\geq l=u\geq 2$ when $l=u$, we can assume without loss of generality that $n\notin S_3$. Then the problem can be reduced to the case when $l=u-1$ by projecting out $z_n$ using $z_n=u-\sum_{j=1}^{n-1}z_j$.
\end{rmk}

\section{Separating valid inequalities}\label{sec:separation}
Before we present the polyhedral description of $\conv(\bXJDd)$ for the general case, we note that the separation problem for inequalities \eqref{eq:ineq1}-\eqref{eq:ineq10}  can be solved efficiently.
We will use this observation in the proof of our main result in Section \ref{sec:general}  as well.

%\begin{linenomath}
Consider a point  $(\hat{\delta},\hat{z})$ that satisfies inequalities \eqref{d0le1}-\eqref{cardconl}, and notice that	
% it is possible to check if the point belongs to $\conv(\bXJDd)$ in polynomial time.	
%Moreover, if the answer is no, then a violated inequality of the form  \eqref{eq:ineq1}-\eqref{eq:ineq10} can be found in polynomial time.	
%This is due to the fact that for a given point, 	
the left-hand side of any inequality of the form \eqref{eq:ineq1}-\eqref{eq:ineq10} is a function of $Q$ that can be written as a constant plus a sum of some additive scores $\pi_j$ for $j\in Q$. Therefore, a most violated inequality can be found  by  choosing the set $Q$ greedily. 	
For example, inequality \eqref{eq:ineq1} can be written as\begin{displaymath}	
u\delta_0+\sum_{j\in Q\cap S_0}z_j+\sum_{i=1}^2\sum_{j\in Q\cap(S_i\setminus S_0)}(z_j-\delta_0+\delta_i)+\sum_{j\in Q\setminus S_3}(z_j-\delta_0)\leq u,	
\end{displaymath}	
and therefore, for a given point  $(\hat{\delta},\hat{z})$ it is possible to define scores  $\pi_j$ for $j\in J$ as follows: \begin{displaymath}	
\pi_j=\begin{cases}	
\hat{z}_j & \text{if }j\in S_0,\\	
\hat{z}_j-\hat{\delta}_0+\hat{\delta}_i & \text{if }j\in (S_i\setminus S_0), i\in\{1,2\},\\	
\hat{z}_j-\hat{\delta}_0 & \text{if }j\in J\setminus S_3.	
\end{cases}	
\end{displaymath}	
The left-hand side of \eqref{eq:ineq1} then becomes $u\delta_0+\sum_{j\in Q}\pi_j$ and it is maximized by choosing	
$$Q\in\arg\max_{Q\subseteq J}\Big\{\sum_{j\in Q}\pi_j:|Q\setminus S_0|\leq u \Big\},$$	
which   can be done greedily  by picking $j\in J$ with the largest positive values $\pi_j$ while satisfying $|Q\setminus S_0|\leq u$.	
The remaining inequalities \eqref{eq:ineq2}-\eqref{eq:ineq10} can be written similarly and in each case a most violated inequality (if there is one) can be identified by first sorting the scores in non-increasing order and then picking the indices $j\in J$ greedily while the cardinality constraint is satisfied. 	
This leads to a polynomial-time algorithm for finding the most violated inequality for each one of  \eqref{eq:ineq1}-\eqref{eq:ineq10}.
%\end{linenomath}

\begin{table}[bt!]
	\begin{center}
		\begin{tabular}{ @{\extracolsep{\fill}} crccc}
			\toprule
			Inequality & \multicolumn{4}{c}{Scores $\pi_j$ for  $j\in J$}\\
			\midrule\midrule\addlinespace
			& $j\in S_0$ & ~~~$j\in S_1\setminus S_0$~~~ & ~~~$j\in S_2\setminus S_0$ ~~~&~~~ $j\in J\setminus S_3$~~~\\
			\cmidrule(lr){2-5}\addlinespace
			\eqref{eq:ineq1} & $\hat{z}_j$ & $\hat{z}_j-\hat{\delta}_0+\hat{\delta}_1$ & $\hat{z}_j-\hat{\delta}_0+\hat{\delta}_2$ & $\hat{z}_j-\hat{\delta}_0$\\
			\eqref{eq:ineq2} & $\hat{z}_j$ & $\hat{z}_j-\hat{\delta}_2+\hat{\delta}_3$ & $\hat{z}_j-\hat{\delta}_1+\hat{\delta}_3$ & $\hat{z}_j-\hat{\delta}_1-\hat{\delta}_2+\hat{\delta}_3$\\
			\eqref{eq:ineq3} with $(i,k)=(1,2)$ & $\hat{z}_j$ & $\hat{z}_j-\hat{\delta}_2$ & $\hat{z}_j$ & $\hat{z}_j-\hat{\delta}_2$\\
			\eqref{eq:ineq3} with $(i,k)=(2,1)$ & $\hat{z}_j$ & $\hat{z}_j$ & $\hat{z}_j-\hat{\delta}_1$ & $\hat{z}_j-\hat{\delta}_1$\\
			\eqref{eq:ineq4} with $i=1$ & $\hat{z}_j$ & $\hat{z}_j$ & $\hat{z}_j-\hat{\delta}_1+\hat{\delta}_3$ & $\hat{z}_j-\hat{\delta}_1$\\
			\eqref{eq:ineq4} with $i=2$ & $\hat{z}_j$ & $\hat{z}_j-\hat{\delta}_2+\hat{\delta}_3$ & $\hat{z}_j$ & $\hat{z}_j-\hat{\delta}_2$\\
			\eqref{eq:ineq5} & $\hat{z}_j$ & $\hat{z}_j$ & $\hat{z}_j$ & $\hat{z}_j-\hat{\delta}_3$\\
			\eqref{eq:ineq6} & $-\hat{z}_j$ & $-\hat{z}_j+\hat{\delta}_0-\hat{\delta}_1$ & $-\hat{z}_j+\hat{\delta}_0-\hat{\delta}_2$ & $-\hat{z}_j+\hat{\delta}_0$\\
			\eqref{eq:ineq7} & $-\hat{z}_j$ & $-\hat{z}_j+\hat{\delta}_2-\hat{\delta}_3$ & $-\hat{z}_j+\hat{\delta}_1-\hat{\delta}_3$ & $-\hat{z}_j+\hat{\delta}_1+\hat{\delta}_2-\hat{\delta}_3$\\
			\eqref{eq:ineq8} with $(i,k)=(1,2)$ & $-\hat{z}_j$ & $-\hat{z}_j+\hat{\delta}_2$ & $-\hat{z}_j$ & $-\hat{z}_j+\hat{\delta}_2$\\
			\eqref{eq:ineq8} with $(i,k)=(2,1)$ & $-\hat{z}_j$ & $-\hat{z}_j$ & $-\hat{z}_j+\hat{\delta}_1$ & $-\hat{z}_j+\hat{\delta}_1$\\
			\eqref{eq:ineq9} with $i=1$ & $-\hat{z}_j$ & $-\hat{z}_j$ & $-\hat{z}_j+\hat{\delta}_1-\hat{\delta}_3$ & $-\hat{z}_j+\hat{\delta}_1$\\
			\eqref{eq:ineq9} with $i=2$ & $-\hat{z}_j$ & $-\hat{z}_j+\hat{\delta}_2-\hat{\delta}_3$ & $-\hat{z}_j$ & $-\hat{z}_j+\hat{\delta}_2$\\
			\eqref{eq:ineq10} & $-\hat{z}_j$ & $-\hat{z}_j$ & $-\hat{z}_j$ & $-\hat{z}_j+\hat{\delta}_3$\\
			\bottomrule
		\end{tabular}
	\end{center}
	\caption{{Scores $\pi_j$ for inequalities \eqref{eq:ineq1}-\eqref{eq:ineq10}} given solution $(\hat{\delta},\hat{z})$}\label{table:score}
\end{table}

For each one of \eqref{eq:ineq1}-\eqref{eq:ineq10}, scores $\pi_j$ for $j\in J$ are summarized in Table \ref{table:score}. An important fact that will be used in proving the main result is that if $(\hat{\delta},\hat{z})$ satisfies $0\leq\hat{\delta}_3\leq\hat{\delta}_i\leq\hat{\delta}_0\leq 1$ for $i\in\{1,2\}$, then for each one of \eqref{eq:ineq1}-\eqref{eq:ineq5}, the score $\pi_j$ is nonpositive for each $j\in J$ satisfying $\hat{z}_j=0$, and for each one of \eqref{eq:ineq6}-\eqref{eq:ineq10}, the score $\pi_j$ is nonnegative for each $j\in J$ satisfying $\hat{z}_j=0$.

\section{General  case}\label{sec:general}
We now consider  the general case when the conditions $S_1\cap S_2\not=\emptyset$ and $|S_1\cup S_2|\le n-l$ do not necessarily hold. 
Given sets $S_1$ and $ S_2$, if either of these conditions is violated, we would define a new instance that satisfies the conditions. We then show that the original convex hull is a face of the convex hull of the modified instance.
%modify the ground set $J$ with additional elements and obtain  the convex hull of the modified instance using the inequalities \eqref{d0le1}-\eqref{eq:ineq10}.
%We then set the $z_j$ variables associated with these additional elements to zero to obtain a face of the convex hull of the modified instance which  gives the  convex hull of the original instance.

We next  present the main result of this paper.
%We next show  that $\conv(\bXJDd)$ can still be described by the same set of inequalities presented earlier.
%\begin{thm}
%	If $S_0=\emptyset$ or $|S_3|>n-l$, then $\conv(\bXJDd)$ is given by inequalities \eqref{d0le1}-\eqref{eq:ineq10} with $\delta_0=1$ if $S_0=\emptyset$ and $\delta_3=0$ if $|S_3|>n-l$.
%\end{thm}

\begin{thm}\label{thm:general}
The convex hull of $\bXJDd$ is given by inequalities \eqref{d0le1}-\eqref{eq:ineq10}.
\end{thm}
\begin{proof}
	%\begin{linenomath}
	Without loss of generality, assume $J=\{1,2,\ldots,n\}$. We augment the base set $J$ and sets $S_i$ for $i\in\{1,2\}$ as follows:\begin{align*}
	J^+=&~\big\{1,2,\ldots,\max(n,|S_3|+l)+k_0 \big\},\\[.2cm]
	S_i^+=&\left\{\begin{array}{ll}
	S_i & \text{if }S_1\cap S_2\neq\emptyset\\
	S_i\cup\{n+1\} & \text{otherwise}
	\end{array}
	\right.~~~~\text{ for }i\in\{1,2\},
	\end{align*}
	where $k_0$ denotes if $S_1\cap S_2=\emptyset$ or not, i.e.,\begin{displaymath}
	k_0=\left\{\begin{array}{ll}
	1 & \text{if }S_1\cap S_2=\emptyset,\\
	0 & \text{otherwise.}
	\end{array}
	\right.
	\end{displaymath}
	Let $\bXJDd_+$ denote the  CCMS associated with this new ground set  $J^+$ and family  $\bar{\cS}_+:=\{S^+_0,S^+_1,S^+_2,S^+_3\}$, where $S^+_0:=S^+_1\cap S^+_2$ and $S^+_3:=S^+_1\cup S^+_2$.	
	By Theorem \ref{thm:1}, the convex hull of $\bXJDd_+$ is given by \eqref{d0le1}-\eqref{eq:ineq10} with $S_i$ replaced by $S_i^+$, $J$ replaced by $J^+$ and $n$ replaced by $n^+:=|J^+|=\max(n,|S_3|+l)+k_0$.
	Notice that 
	$$ (\delta_0,\delta_1,\delta_2,\delta_3,\{z_j\}_{j\in J})\in \bXJDd ~\Longleftrightarrow~(\delta_0,\delta_1,\delta_2,\delta_3,\{z_j\}_{j\in J^+})\in \bXJDd_+ \text{ where } z_j=0\text{ for } j\in J^+\setminus J,$$
	as $\delta_i=\prod_{j\in S_i^+}(1-z_j)=\prod_{j\in S_i}(1-z_j)$ when $z_j=0$ for $j\in J^+\setminus J$. 
 Therefore,
	$$ \bXJDd \times \{0\}^{(n^+-n)}~=~ \bXJDd_+\cap\big\{    (\delta,z) \in\R^4\times \R^{n_+}\::\: z_j=0,~j\in J^+\setminus J\big\}$$
	and, consequently,
	\begin{equation}\label{conv:extended}
	\conv(\bXJDd) \times \{0\}^{(n^+-n)}~=~ \conv(\bXJDd_+)\cap\big\{    (\delta,z) \in\R^4\times \R^{n_+}\::\: z_j=0,~j\in J^+\setminus J\big\}.
	\end{equation}
	In other words, $\conv(\bXJDd) $ is a proper face of $\conv(\bXJDd_+)$ obtained by setting the additional variables $z_{n+1},\ldots, z_{n^+}$ to zero.
	This implies that $\conv(\bXJDd)$ can be described in the extended space by inequalities \eqref{d0le1}-\eqref{eq:ineq10} with $S_i$ replaced by $S_i^+$, $J$ replaced by $J^+$ and $n$ replaced by $n^+$ together with equations  $z_j=0$ for $j\in J^+\setminus J$. 
	Note that, 	since $z_j=0$ for $j\in J^+\setminus J$  in this description, inequalities \eqref{d0le1} and \eqref{zjdige1} imply that $\delta_0=1$ if $S_0=\emptyset$, and inequalities \eqref{d3ge0} and \eqref{eq:ineq10} with $\Sprime =\emptyset$ imply that $\delta_3=0$ if $|S_3|>n-l$.
	In addition, since $0\le\delta_3\leq\delta_i\leq\delta_0\le1$ for $i\in\{1,2\}$ by inequalities \eqref{d0le1}-\eqref{dineq},  any inequality  \eqref{eq:ineq1}-\eqref{eq:ineq5} given by a set $\Sprime\subseteq J^+$ is implied by the same inequality given by the set $\Sprime\cap J$. In particular, given $(\hat{\delta},\hat{z})$ satisfying $\hat{z}_j=0$ for $j\in J^+\setminus J$ and $0\le\hat{\delta}_3\leq\hat{\delta}_i\leq\hat{\delta}_0\le1$ for $i\in\{1,2\}$, the scores $\pi_j$ for all $j\in J^+\setminus J$ of each class of inequalities \eqref{eq:ineq1}-\eqref{eq:ineq5} are nonpositive. Therefore, inequalities \eqref{eq:ineq1}-\eqref{eq:ineq5} with $Q\cap (J^+\setminus J)\neq\emptyset$ can never be the most violated inequality provided that $\hat{z}_j=0$ for $j\in J^+\setminus J$ and $0\le\hat{\delta}_3\leq\hat{\delta}_i\leq\hat{\delta}_0\le1$ for $i\in\{1,2\}$ are satisfied, which implies that only inequalities \eqref{eq:ineq1}-\eqref{eq:ineq5} with $Q\subseteq J$ can be non-redundant for describing \eqref{conv:extended} given \eqref{d0le1}-\eqref{dineq}.
	Similarly, any inequality  \eqref{eq:ineq6}-\eqref{eq:ineq10} given by a set $\Sprime\subseteq J^+$ is implied by the same inequality given by $\Sprime\cup (J^+\setminus J)$.  In other words, only inequalities \eqref{eq:ineq6}-\eqref{eq:ineq10} with $Q\supseteq J^+\setminus J$ can be non-redundant for describing \eqref{conv:extended} given \eqref{d0le1}-\eqref{dineq}.
	 It follows that $\conv(\bXJDd)$ is characterized by inequalities \eqref{d0le1}-\eqref{eq:ineq10} in the original space.
\end{proof}
\begin{rmk}
	Similar to Remark \ref{rmk:1}, Theorem \ref{thm:general} holds even when $l=u$ as the proof does not depend on the assumption that $l<u$.
\end{rmk}

%\begin{linenomath}
We note that when $S_0=\emptyset$, inequalities  \eqref{eq:ineq1} and \eqref{eq:ineq6} are redundant. 
In this case, setting $\delta_0=1$, inequality \eqref{eq:ineq1} becomes 
$$\sum_{j\in \Sprime }z_j+\sum_{i=1}^2(|\Sprime \cap S_i|)\delta_i\leq |\Sprime |,$$ 
which is implied by \eqref{zjdile1} and \eqref{zjle1}. 
Similarly, inequality \eqref{eq:ineq6} becomes 
$$-\sum_{j\in \Sprime }z_j+\sum_{i=1}^2(|S_i\setminus \Sprime |)\delta_i\leq n-l-|\Sprime |,$$ 
which is implied by \eqref{zjdile1}, \eqref{zjle1} and \eqref{cardconl}. 
%\end{linenomath}

On the other hand, when $|S_3|>n-l$, inequalities \eqref{eq:ineq4}, \eqref{eq:ineq5}, \eqref{eq:ineq9} and \eqref{eq:ineq10} become redundant. 
In this case, since $\delta_3=0$, inequalities \eqref{eq:ineq2} and $\delta_i\geq 0$ imply inequality \eqref{eq:ineq4}.
Similarly, inequalities \eqref{zjge0} and \eqref{cardconu} imply \eqref{eq:ineq5}, inequalities \eqref{eq:ineq7} and $\delta_i\geq 0$ imply \eqref{eq:ineq9}, and inequality \eqref{zjge0} implies \eqref{eq:ineq10}.

We should also note that inequalities \eqref{eq:ineq1}-\eqref{eq:ineq5} are redundant when $u=n$, and inequalities \eqref{eq:ineq6}-\eqref{eq:ineq10} are redundant when $l=0$. For example, given $(\hat{\delta},\hat{z})$ satisfying \eqref{d0le1}-\eqref{cardconl}, if $u=n$, then the left-hand side of \eqref{eq:ineq1} becomes $\sum_{j\in Q}\pi_j+n\hat{\delta}_0$. Note that for \eqref{eq:ineq1}, $\pi_j+\hat{\delta}_0\leq 1$ for each $j\in J$. Therefore, $\sum_{j\in Q}\pi_j+n\hat{\delta}_0\leq\sum_{j\in J}(\pi_j+\hat{\delta}_0)\leq n$, i.e., \eqref{eq:ineq1} can never be violated if $u=n$ and \eqref{d0le1}-\eqref{cardconl} are satisfied. Similar arguments can be made for each of \eqref{eq:ineq1}-\eqref{eq:ineq5} when $u=n$ and for each of \eqref{eq:ineq6}-\eqref{eq:ineq10} when $l=0$.

Observe that the validity of inequalities \eqref{d0le1}-\eqref{eq:ineq10} does not depend on the fact that $S_1\setminus S_2\neq\emptyset$ or $S_2\setminus S_1\neq\emptyset$. In fact, these inequalities describe the convex hull even when $\cS$ is nested. For example, when $S_1\subset S_2$, we have $\delta_0=\delta_1$ and $\delta_3=\delta_2$ by \eqref{d3ledi}, \eqref{dineq} and \eqref{2link}. Then inequalities \eqref{d0le1}-\eqref{eq:ineq10}, with $\delta_0$ replaced by $\delta_1$ and $\delta_3$ replaced by $\delta_2$, coincide with the inequality description of $\conv(\XJDd)$ for the nested case \cite{chen2020multilinear}.

%Note that when $S_0=\emptyset$, \eqref{eq:ineq1} and \eqref{eq:ineq6} are redundant. In this case, since $\delta_0=1$, \eqref{eq:ineq1} is equivalent to $\sum_{j\in S'}z_j+\sum_{i=1}^2(|S'\cap S_i|)\delta_i\leq |S'|$, which is implied by \eqref{zjdile1} and \eqref{zjle1}, and \eqref{eq:ineq1} is equivalent to $-\sum_{j\in S'}z_j+\sum_{i=1}^2(|S_i\setminus S'|)\delta_i\leq n-l-|S'|$, which is implied by \eqref{zjdile1}, \eqref{zjle1} and \eqref{cardconl}. When $|S_3|>n-l$, \eqref{eq:ineq4}, \eqref{eq:ineq5}, \eqref{eq:ineq9} and \eqref{eq:ineq10} are redundant. In that case, since $\delta_3=0$, \eqref{eq:ineq2} and $\delta_i\geq 0$ imply \eqref{eq:ineq4}, \eqref{cardconu} and \eqref{zjge0} imply \eqref{eq:ineq5}, \eqref{eq:ineq7} and $\delta_i\geq 0$ imply \eqref{eq:ineq9}, and \eqref{zjge0} implies \eqref{eq:ineq10}.

\section{Conclusions}
We have characterized the convex hull of CCMS-2 using an extended formulation. The analysis significantly relies on the characterization of type-1 and type-2 inequalities for CCMSs associated with proper families. It is natural to ask if our results can be extended to describe the convex hull of CCMSs with fixed number of $\delta_i$ variables. By using Balas' disjunctive model \cite{balas1979disjunctive}, we can give an extended formulation of the convex hull of such CCMSs \cite{chen2020multilinear}. However, a polyhedral description for the set in the $(\delta,x)$ space is not clear yet for $|\cS|>2$.
\newpage
\bibliographystyle{plain}%ieeetr
\bibliography{ref}

\begin{thebibliography}{10}

\bibitem{balas1979disjunctive}
Egon Balas.
\newblock Disjunctive programming.
\newblock {\em Annals of Discrete Mathematics}, 5:3--51, 1979.

\bibitem{bk}
Christoph Buchheim and Laura Klein.
\newblock Combinatorial optimization with one quadratic term: spanning trees
  and forests.
\newblock {\em Discrete Applied Mathematics}, 177:34--52, 2014.

\bibitem{chen2020cardinality}
Rui Chen, Sanjeeb Dash, and Oktay G{\"u}nl{\"u}k.
\newblock Cardinality constrained multilinear sets.
\newblock In {\em International Symposium on Combinatorial Optimization}, pages
  54--65. Springer, 2020.

\bibitem{chen2020multilinear}
Rui Chen, Sanjeeb Dash, and Oktay Gunluk.
\newblock Convexifying multilinear sets with cardinality constraints:
  Structural properties, nested case and extensions.
\newblock {\em arXiv preprint arXiv:2007.15725}, 2020.

\bibitem{crama1993concave}
Yves Crama.
\newblock Concave extensions for nonlinear 0--1 maximization problems.
\newblock {\em Mathematical Programming}, 61(1):53--60, 1993.

\bibitem{crama2017class}
Yves Crama and Elisabeth Rodr{\'\i}guez-Heck.
\newblock A class of valid inequalities for multilinear 0--1 optimization
  problems.
\newblock {\em Discrete Optimization}, 25:28--47, 2017.

\bibitem{dash2018boolean}
Sanjeeb Dash, Oktay Gunluk, and Dennis Wei.
\newblock Boolean decision rules via column generation.
\newblock {\em Advances in neural information processing systems}, 31, 2018.

\bibitem{del2017polyhedral}
Alberto Del~Pia and Aida Khajavirad.
\newblock A polyhedral study of binary polynomial programs.
\newblock {\em Mathematics of Operations Research}, 42(2):389--410, 2017.

\bibitem{del2018multilinear}
Alberto Del~Pia and Aida Khajavirad.
\newblock The multilinear polytope for acyclic hypergraphs.
\newblock {\em SIAM Journal on Optimization}, 28(2):1049--1076, 2018.

\bibitem{del2018decomposability}
Alberto Del~Pia and Aida Khajavirad.
\newblock On decomposability of multilinear sets.
\newblock {\em Mathematical Programming}, 170(2):387--415, 2018.

\bibitem{del2021running}
Alberto Del~Pia and Aida Khajavirad.
\newblock The running intersection relaxation of the multilinear polytope.
\newblock {\em Mathematics of Operations Research}, 2021.

\bibitem{dbs}
Ayhan Demiriz, Kristin~P. Bennett, and John Shawe-Taylor.
\newblock Linear programming boosting via column generation.
\newblock {\em Machine Learning}, 46:225--254, 2002.

\bibitem{dgm}
David~P. Dobkin, Dimitrios Gunopulos, and Wolfgang Maass.
\newblock Computing the maximum bichromatic discrepancy, with applications to
  computer graphics and machine learning.
\newblock {\em Journal of Computer and Systems Sciences}, 52:453--470, 1996.

\bibitem{eg2}
Jonathan Eckstein and Noam Goldberg.
\newblock An improved branch-and-bound method for maximum monomial agreement.
\newblock {\em INFORMS Journal on Computing}, 24(2):328--341, 2012.

\bibitem{eg3}
Jonathan Eckstein, Ai~Kagawa, and Noam Goldberg.
\newblock Repr: Rule-enhanced penalized regression.
\newblock {\em INFORMS Journal on Optimization}, 1(2):143--163, 2019.

\bibitem{ff}
Anja Fischer and Frank Fischer.
\newblock Complete description for the spanning tree problem with one
  linearised quadratic term.
\newblock {\em Operations Research Letters}, 41:701--705, 2013.

\bibitem{fischer2018matroid}
Anja Fischer, Frank Fischer, and S~Thomas McCormick.
\newblock Matroid optimisation problems with nested non-linear monomials in the
  objective function.
\newblock {\em Mathematical Programming}, 169(2):417--446, 2018.

\bibitem{fomeni2017new}
Franklin~Djeumou Fomeni.
\newblock A new family of facet defining inequalities for the maximum
  edge-weighted clique problem.
\newblock {\em Optimization Letters}, 11(1):47--54, 2017.

\bibitem{fortet}
R.~Fortet.
\newblock Applications de l'alg{\`e}bre de boole en recherche
  op{\'e}rationelle.
\newblock {\em Revue Francaise de Recherche Operationelle}, 4:17--26, 1960.

\bibitem{hosseinian2017maximum}
Seyedmohammadhossein Hosseinian, Dalila~BMM Fontes, Sergiy Butenko,
  Marco~Buongiorno Nardelli, Marco Fornari, and Stefano Curtarolo.
\newblock The maximum edge weight clique problem: formulations and solution
  approaches.
\newblock In {\em Optimization Methods and Applications}, pages 217--237.
  Springer, 2017.

\bibitem{mehrotra1997cardinality}
Anuj Mehrotra.
\newblock Cardinality constrained boolean quadratic polytope.
\newblock {\em Discrete Applied Mathematics}, 79(1-3):137--154, 1997.

\bibitem{padberg1989boolean}
Manfred Padberg.
\newblock The boolean quadric polytope: some characteristics, facets and
  relatives.
\newblock {\em Mathematical programming}, 45(1-3):139--172, 1989.

\bibitem{sorensen2004new}
Michael~M S{\o}rensen.
\newblock New facets and a branch-and-cut algorithm for the weighted clique
  problem.
\newblock {\em European Journal of Operational Research}, 154(1):57--70, 2004.

\end{thebibliography}
\end{document}